\documentclass[12pt,a4paper,draft,twoside,reqno]{amsart}

\usepackage{graphicx}

\usepackage{mathptmx}

\usepackage{amsfonts, amsmath, amssymb,euscript,amscd}

\DeclareFontFamily{OML}{cyr}{} \DeclareFontShape{OML}{cyr}{m}{n}{
  <5> <6> <7> <8> <9> gen * wncyr
  <10> <10.95> <12> <14.4> <17.28> <20.74> <24.88> wncyr10
  }{}
\DeclareSymbolFont{rusletters}{OML}{cyr}{m}{n}
\DeclareSymbolFontAlphabet{\rusmath}{rusletters}
\DeclareMathSymbol\re{\rusmath}{rusletters}{"03}

\newtheorem{theorem}{Theorem}[section]
\newtheorem{proposition}[theorem]{Proposition}
\newtheorem{lemma}[theorem]{Lemma}
\newtheorem{corollary}[theorem]{Corollary}

\newcommand{\id}{\mathop{\rm id}\nolimits} % Identical map
\newcommand{\Orb}{\mathop{\rm Orb}\nolimits}

\newcommand{\E}{\EuScript E} % Equation
\newcommand{\R}{\mathbb R} % Real numbers
\newcommand{\eq}{y''=a^3(x, y)(y')^3+a^2(x, y)(y')^2+a^1(x,y)y'+a^0(x, y)}

\begin{document}

\title[Differential invariants  and equivalence]{Differential invariants  and equivalence of ODEs $\mathbf{y''=a^3(x,y)y'^3+a^2(x,y)y'^2+a^1(x,y)y'+a^0(x,y)}$}  

\author{Valeriy A. Yumaguzhin}

\email{yuma@diffiety.botik.ru} 

\date{20 January 2025}

\begin{abstract} This paper is devoted to ordinary differential equations of the form 
$$
  \eq.
$$ 
The algebra of all differential invariants of point transformations is constructed for these  equations in general position and the equivalence problem is solved for them.
\end{abstract}

\keywords{2-nd order ordinary differential equation, point transformation, differential invariant}

\subjclass{53A55, 53C10, 53C15, 34A30, 34A26, 34C20, 58F35}

\maketitle

%%%%%%%%%%%%%%%%%%%%%%%%%%%%%%%%%%%%%%%%%%
\section{Introduction}
This paper is devoted to differential invariants  of point transformations of equations 
\begin{equation}\label{eq}
 \eq
\end{equation}
in general position.
 
There are different approaches to construct differential invariants of these equations:  E. Cartan \cite{Crtn}, R.B. Gardner \cite{Grdnr}, S. Lie \cite{Li1,Li2}, R. Liouville \cite{Lvll}, G. Thomsen \cite{Thmsn}, A. Tresse \cite{Trss1,Trss2},  R. Sharipov \cite{Shrpv1, Shrpv2}, and  V. Yumaguzhin \cite{Yum2}.

In work \cite{Yum2} a complete family of generators of the algebra of all differential invariants under consideration was constructed for the first time. In addition, a complete family of all differential syzygies of these generators are obtained. 
\smallskip

In this paper, it is represented: a direct description of this algebra and a solution of the equivalence problem of equations \eqref{eq} in general position under point transformations. 

%=======================================================
\subsection{Preliminaries}
\subsubsection{Notations}
Let $M$ be a connected smooth $n$-dimensional manifold,  $x_1,\ldots, x_n$ be local coordinates in $M$, $x=(x_1,\ldots, x_n)$, and ${\mathcal G(M)}$ be a pseudogroup of all diffeomorphisms of $M$.

Let $\pi\colon E\rightarrow M$ be a $m$-dimensional vector bundle over $M$, $x_1,\ldots, x_n$, $u^1$, $\ldots$, $u^m$ be local coordinates in $\pi$, and $C^{\infty}(\pi)$ be the $C^{\infty}(M)$-module of smooth sections of $\pi$.

Let $\pi_k\colon J^k\pi\rightarrow M$ be the bundle of $k$-jets of sections of $\pi$,  
$\pi_{k,l} \colon J^k\pi\rightarrow J^l\pi$,\; $k> l$, be the projection defined by reductions of $k$-jets to their $l$-jets, and $x_1, \dots, x_n$,  $u^i_{\sigma}$, $i=1,\ldots,m$, $0\leq |\sigma|\leq k$,  where $\sigma$ is the multi-index $\{j_1\ldots j_r\}$, $|\sigma|=r$,\; $j_1,\ldots,j_r=1,\ldots, n$, be canonical coordinates in $J^k\pi$.
\medskip
 
 Any sections $S\in C^{\infty}(\pi)$ defines its $k$-jet prolongation $j_kS\in C^{\infty}(\pi_k)$ by the formula $j_kS(a)=j_a^kS$.
%--------------------------------------------------------------------------------------------------------------
\subsubsection{Natural bundles} 
A bundle $\pi$ is called {\it natural} if for each $f\in {\mathcal G(M)}$ there is defined a diffeomorphism $f^{(0)}\colon E\rightarrow E$ such that the following conditions are satisfied:
\begin{enumerate}
 \item  $\pi\circ f^{(0)}=f\circ\pi$,
 \item  $\id_M^{(0)}=\id_E$, where $\id_M$ and $\id_E$ are identity maps of $M$ and $E$ respectively,
 \item  $(f\circ g)^{(0)}=f^{(0)}\circ g^{(0)}$\;  for all $f, g\in {\mathcal G(M)}$,
\end{enumerate}

\begin{corollary}
If $\pi$ is a natural bundle, then 
for each $f\in {\mathcal G(M)}$ its diffeomorphism $f^{(0)}\colon E\rightarrow E$ 
 is uniquely determined.
\end{corollary}
For this reason we will call this  unique diffeomorphism $f^{(0)}$ as the natural lift of $f$ to the bundle $\pi$.

A section of a natural bundle is called {\it a differential structure} (on the base of $\pi$), see \cite{ALV}.
\medskip 

{\it Differential order} of a natural bundle is the smallest natural number $r$  such that for any $f\in\mathcal G(M)$ the value $f^{(0))}(\theta_0)$ at a point $\theta_0\in E$, where $p=\pi(\theta_0)$ is a point from domain of $f$, is determined by the $r$-jet $j_p^rf$.

%--------------------------------------------------------------------------------------------------------------
\subsubsection{Natural bundles of ODEs \eqref{eq}} 
In the rest of this article we will denote by $\pi$ the following vector bundle
$$  
  \pi\colon E \longrightarrow M\quad \pi\colon ( p,\,u^1,\,\dots ,\, u^4\,)\mapsto p, \quad p=(x, y),
$$  
where $M=\R^2$, $E=M\times\R^4$, $x, y$ are standard coordinates on $M$, $u^1,\dots, u^4$ are standard coordinates on the fiber $\R^4$ of $\pi$. 

Let $\E$ be an arbitrary ODE \eqref{eq}. Following \cite{Yum2}, we will consider every equation $\E$ as a geometric structure. Namely, identify every equation $\E$ with the section $S_{\E}$ of $\pi$
\begin{equation}\label{EqSctn}
  S_{\E}\colon M\longrightarrow E,\quad  S_{\E}\colon p\mapsto \big(\,p,\,a^0(p),\,\dots ,\,  a^3(p)\,\big).
\end{equation}

As a result, the set of all equations $\E$ is identified with the set $C^{\infty}(\pi)$ of all sections $S_{\E}$ of $\pi$. 
\medskip

It is well known, \cite{Arnld}, that each diffeomorphism  $f\in\mathcal G(M),\; f\colon (x,y)\mapsto \big(\tilde x(x,y), \tilde y(x,y)\big) $ generates the transformation of every equation $\E$ to the equation $\tilde\E$ of the same form:
$$ 
\tilde y''=\tilde a^3(\tilde x,\tilde y)(\tilde y')^3+\tilde a^2(\tilde x,\tilde y)(\tilde y')^2+\tilde a^1(\tilde x,\tilde y)\tilde y'+\tilde a^0(\tilde x, \tilde y).
$$
The coefficients of $\tilde\E$ are expressed in terms of the coefficients of $\E$ and the $2$-jets of the inverse transformation $f^{-1}$ by formulas:
\begin{equation}\label{CffTrnsfrm}
 \tilde a^i(\tilde p)=\Phi^i\bigl(\,a^0(f^{-1}(\tilde p)),\ldots,a^3(f^{-1}(\tilde p)),\,
 j^2_{\tilde p}f^{-1}\,\bigr),\quad i=0,1,2,3.
\end{equation}
It follows that the equations
\begin{equation}\label{Lft0}   
  \tilde p=f(p),\quad 
   \tilde u^i=\Phi^{i-1}\bigl(\,u^1,\ldots,u^4,\,j^2_{f(p)}f^{-1}\bigr),\,\; i=1,\ldots,4
\end{equation}
define the diffeomorphism $f^{(0)}$ of the total space $E$ of $\pi$. 

These diffeomorphisms $f^{(0)}$ satisfy the conditions of naturalness: 
$$
\pi\circ f^{(0)}=f\circ\pi,\quad \id_M^{(0)}=\id_E,\quad
(f\circ h)^{(0)}=f^{(0)}\circ h^{(0)},\;\; f,h\in\mathcal{G}(M). 
$$  

Thus, the bundle $\pi$ is a natural bundle (of equations $\E$),  and every equation $\E$, considering as a section of $\pi$, is a geometric structure on $M$. 

It follows from \eqref{Lft0} that differential order of the natural bundle $\pi$ is 2.

%--------------------------------------------------------------------------------------------------------
\subsubsection{Natural bundles of jets of ODEs \eqref{eq}} 

Equations (\ref{Lft0}), represented in the terms of the corresponding sections, has the form:  
$$
    S_{\tilde\E}=f^{(0)}\circ S_{\E}\circ f^{-1},\quad f\in\mathcal G(M).
$$

Every $f\in\mathcal G(M)$ is lifted to the diffeomorphism $f^{(k)}$ of the bundle $J^k\pi$ by the formula
\begin{equation}\label{LftTr}
f^{(k)}(\,j^k_pS\,)=j^k_{f(p)}\bigl(f^{(0)}\circ S\circ f^{-1}\bigr),\quad k=0,1,2,\ldots  .
\end{equation}

\begin{lemma}
The bundle $\pi_k\colon J^k\pi\rightarrow M$ with lifting \eqref{LftTr} of diffeomorphisms $\mathcal G(M)$ in the bundle $\pi_k$ is natural.
\end{lemma}

%---------------------------------------------------------------------------------------------------------------
\subsubsection{Orbits, equations, and differential invariants in general position} 
As a result of the action $\mathcal G(M)$ on $J^k\pi$, the latter is divided into orbits $\Orb^k_r$, where $r$ is the dimension of an isotropy algebra of point of $\Orb^k_r$.

We will call an orbit of codimension zero as {\it an orbit of general position}, other ones we will call as {\it degenerate orbits}. 

We get from  \cite{Yum2}:
\begin{itemize}
 \item The bundles $J^0\pi$ and $J^1\pi$ are orbits of the action of $\mathcal{G}(\R^2)$.
 \item  The bundle $J^2\pi$ is the union of two orbits: $\Orb^2_4$ is an orbit of general position and $\Orb^2_6$  is a degenerate orbit. The first one is defined by the inequality $(L_1, L_2)\neq  0$, where $L_1, L_2$ are defined by equations \eqref{L1L2}.
 \item The bundle $J^3\pi$ is the union of four orbits: $\Orb^3_0$, $\Orb^3_1$, $\Orb^3_2$, and $\Orb^3_6$. The first one is an orbit of general position, it is defined by  inequalities $(L_1, L_2)\neq 0$ and $L_3\neq 0$, where $L_3$ is defined by equation \eqref{L3}, the remaining orbits are degenerate.
 \item The bundle $J^k\pi$, $k\geq 4$, is union of degenerate orbits.  
\end{itemize}
\smallskip

By {\it a differential invariant of order $k$}, $k=0, 1, 2, \ldots$, we mean a smooth function on $J^k\pi$ invariant with respect to all lifted diffeomorphisms $f^{(k)}$, $f\in\mathcal G(M)$. 
\smallskip

By {\it a differential invariant of order $k$ of the equation $\E$} we mean the  restrictions of some differential invariant of order $k$ on the graph of the section $j_kS_{\E}$. 
\smallskip

Let $I_1$ and $I_2$ be differential invariants of order $k$. We will say that they are {\it in  general position in a domain $O\subset J^k\pi$}, if $\widehat d\,I_1\wedge \widehat d\, I_2 \neq 0$ in this domain.
\smallskip

By {\it a tensor differential invariant of order $k$} we mean a smooth tensor field on $J^k\pi$ invariant with respect to all lifted diffeomorphisms $f^{(k)}$, $f\in\mathcal G(M)$. 

%-------------------------------------------------------------------------------------------------------------------
\subsubsection{Total differentials}\label{TD}
Let $f$ be a smooth function on the manifold $J^k\pi$. Then its {\it total differential} $\widehat d f$ is the horizontal differential 1-form on the manifold $J^{k+1}\pi$, which satisfies the following condition
\begin{equation}\label{TDff0}       
   (j_{k+1}S)^*(\widehat d f)=d\bigl( (j_kS)^*(f)\bigr)
\end{equation} 
for all $S\in C^{\infty}(\pi)$.

In canonical coordinates the total differential has the form
\begin{equation}\label{TDff} 
\widehat d f=\frac{df}{dx}dx + \frac{df}{dy}dy,
\end{equation} 
where $\displaystyle\frac{d}{d x}=\frac{\partial}{\partial x}+u_i^j\frac{\partial}{\partial u^j}+\ldots+u^j_{x\, \sigma}\frac{\partial}{\partial u^j_{\sigma}}+\ldots$ and 
$\displaystyle\frac{d}{d y}=\frac{\partial}{\partial y}+u_i^j\frac{\partial}{\partial u^j}+\ldots+u^j_{\sigma\,y}\frac{\partial}{\partial u^j_{\sigma}}+\ldots$.
are {\it total derivatives}.

%--------------------------------------------------------------------------------------------------------------
\subsubsection{Tresse derivatives}
Let differential invariants $I_1$ and $I_2$ be in general position. Then {\it Tresse derivatives} $\displaystyle\frac{df}{dI_i}$, $i=1,2$, where $f\in C^{\infty}(J^{\infty}(\pi))$, are defined by the relation
$$
\displaystyle\widehat d f=\frac{df}{dI_i}\widehat d I_i.
$$

%======================================================
\subsection{Tensor differential invariants} In these section, we introduce necessary for us tensor differential invariants. Their detailed computations can be found in \cite{Yum2}.
  
%----------------------------------------------------------------------------------------------------------------  
\subsubsection{Invariant vector fields} The following vector fields are tensor differential invariants defined on $\Orb^3_0$, (see \cite{Yum2}, p. 302, ref. (49))
\begin{equation}\label{xi1xi2}
  \xi_1=\frac{1}{(L_3)^{2/5}}(L_2\frac{\partial}{\partial x}-L_1\frac{\partial}{\partial y}),\quad
  \xi_2=\frac{1}{(L_3)^{4/5}}(\Psi_2\frac{\partial}{\partial x}-\Psi_1\frac{\partial}{\partial y}),
\end{equation}
where $L_1$, $L_2$, $L_3$, and $\Psi_1$, $\Psi_2$ are defined  by the following equations (see \cite{Yum2}, p. 291, ref. (26); p. 292, ref. (28); p. 301 )
\begin{equation}\label{L1L2}  
 \begin{aligned}
   L_1  = 3u^1_{yy}-2u^2_{xy}+u^3_{xx}\qquad\phantom{+3u^4u^1_{x}-3u^3u^1_{y}+2u^2u^2_{y}\qquad}\\
   +3u^4u^1_{x}-3u^3u^1_{y}+2u^2u^2_{y} 
   - u^2u^3_{x}-3u^1u^3_{y}+6u^1u^4_{x}, \\
   L_2 = 3u^4_{xx}-2u^3_{xy}+u^2_{yy}\qquad\phantom{+3u^4u^1_{x}-3u^3u^1_{y}+2u^2u^2_{y}\qquad}\\ 
   -3u^1u^4_{y}+3u^2u^4_{x}-2u^3u^3_{x}+u^3u^2_{y}+ 3u^4u^2_{x}-6u^4u^1_{y},
 \end{aligned}
\end{equation}
\begin{multline}\label{L3} 
\quad L_3  =  L_2(L_1\frac{dL_2}{dx}-L_2\frac{dL_1}{dx})-L_1(L_1\frac{dL_2}{dy}-L_2\frac{dL_1}{dy})\\
         +L_1^3u^4-L_1^2L_2u^3+L_1L_2^2u^2-L_2^3u^1,
\end{multline}
\begin{equation}\label{Psi1Psi2}
 \begin{aligned} 
  \Psi_1 = - L_1^2u^3 + 2L_1L_2u^2 - 3L_2^2u^1 - L_1\frac{dL_1}{dy} + 4L_1\frac{dL_2}{dx} -3L_2\frac{dL_1}{dx}, \\ 
  \Psi_2 =  - L_2^2u^2 + 2L_1L_2u^3  - 3L_1^2u^4+L_2\frac{dL_2}{dx} - 4L_2\frac{dL_1}{dy} +3L_1\frac{dL_2}{dy}.
 \end{aligned}
\end{equation}
\medskip

The fields $\xi_1$ and $\xi_2$ are linear independent in every point $\theta_3\in\Orb^3_0$ (see \cite{Yum2}, p. 302, Proposition 15).

These fields are known tensor differential invariants, see \cite{Shrpv1, Shrpv2}.

%-----------------------------------------------------------------------------------------------------------
\subsubsection{The invariant volume form} The following tensor differential invariant on $\Orb^3_0$ (see \cite{Yum2}, p. 301, ref. (47) )
\begin{equation}\label{nu}
 \nu=L_3^{1/5}(dx_1\wedge dx_2),
\end{equation}
was first obtained by R. Liouville in \cite{Lvll}.

%%%%%%%%%%%%%%%%%%%%%%%%%%%%%%%%%%%%%%%%%%
\section{The algebra of all differential invariants of ODEs \eqref{eq}} 

%======================================================
\subsection{Coordinates generated by invariants} Taking into account that $\xi_1$, $\xi_2$, and $\nu$ are tensor  differential invariants of order 3 defined on $\Orb_0^3$,  we get that Lie derivatives $\xi_1(\nu)$, $\xi_2(\nu)$ are tensor differential invariants of order 4 defined on $\pi_{4,3}^{-1}(\Orb_0^3)$. These invariants define differential invariants $I_1$ and  $I_2$ of order 4 by the formulas:

$$
  I_1= \frac{\xi_1(\nu)}{\nu},\quad I_2= \frac{\xi_2(\nu)}{\nu}  
$$
which are also defined on  $\pi_{4,3}^{-1}(\Orb_0^3)$.

Let $O$ be a domain in $\pi_{4,3}^{-1}(\Orb^3_0)$, where the invariants $I_1$ and $I_2$ are in general position, i.e. $\widehat d\,I_1\wedge \widehat d\, I_2 \neq 0$ in this domain.

Let $S_{\E}\in C^{\infty}(\pi)$ and $O'$ be a domain in $M$ such that 
$$
  j_4S_{\E}(O')\subset O.
$$ 
Then we will say that {\it the equation $\E$ (the section $S_{\E}$) is in general position in $O'$}. 

The function
$$
  (j_4S_{\E})^*I_i=I_i\circ j_4S_{\E},\quad i=1,2,
$$ 
is the restriction of invariant $I_i$ on the graph of section $j_4S_{\E}$. Thus it is 
a differential invariants of order 4 of the equation $\E$. 
\begin{proposition} 
The differential invariants of $\E$: $(j_4S_{\E})^*I_1$,\;$(j_4S_{\E})^*I_2$, are  independent coordinates on $O'$.
\end{proposition}
\begin{proof}
General position of the invariants $I_1$, $I_2$ means that $\widehat d I_1\wedge \widehat d I_2\neq 0$. It follows that 
\begin{multline*}
(j_5S_{\E})^*(\widehat d I_1\wedge \widehat d I_2)=(j_5S_{\E})^*\widehat d I_1\wedge
(j_5S_{\E})^*\widehat d I_2=\\ d \big( (j_4S_{\E})^*I_1\big)\wedge d \big( (j_4S_{\E})^*I_2\big))\neq 0. 
\end{multline*}
The last inequality means that the differential invariants of $\E$: $(j_4S_{\E})^*I_1$ and $(j_4S_{\E})^*I_2$ 
are independent coordinates on the domain $O'\subset M$. 
\end{proof}

Let us denote these coordinates as $\tilde x=(j_4S_{\E})^*I_1$ and $\tilde y=(j_4S_{\E})^*I_2$. 

\begin{corollary}\label{2.2}  Let $(j_4S_{\E})^*(I_1, I_2)_=\big((4S_{\E})^*I_1, (4S_{\E})^*I_2\big)$. 
Then the mapping 
\begin{equation}\label{g} 
  g= (j_4S_{\E})^*(I_1, I_2)\colon O' \longrightarrow O',\;\; g\colon (x, y)\mapsto \big(\tilde x(x, y), \tilde y(x, y)\big)
\end{equation}
 is a diffeomorphism.
\end{corollary}

In the coordinates $\tilde x, \tilde y$, the ODE $\E$ has the form $\tilde\E$:
\begin{equation}\label{tildeE}  
 \tilde y''=\tilde J_3(\tilde x,\tilde y)(\tilde y')^3+\tilde J_2(\tilde x,\tilde y)(\tilde y')^2+\tilde J_1(\tilde x,\tilde y)\tilde y'+\tilde J_0(\tilde x, \tilde y).
\end{equation}  
where the coefficients $\tilde J_3(\tilde x,\tilde y), \ldots, \tilde J_0(\tilde x,\tilde y)$ are differential invariants. The section $S_{\tilde\E}$, see \eqref{EqSctn}, has the form  
$$
S_{\tilde\E}=\big(\, \tilde J_3(\tilde x,\tilde y),\, \tilde J_2(\tilde x,\tilde y),\, \tilde J_1(\tilde x,\tilde y),\, \tilde J_0(\tilde x,\tilde y) \,\big).
$$

Let $g$ be the diffeomorphism defined by \eqref{g}. Then 
$$
   S_{\tilde\E}=g^{(0)}\circ S_{\E}\circ g^{-1}.
$$

\begin{proposition}\label{InvrntSctn}  
The section $S_{\tilde\E}$ is invariant w.r.t. the action of $\mathcal G(M)$ on $C^{\infty}(\pi)$, i.e. $f^0\circ S_{\tilde\E}\circ f^{-1}=S_{\tilde\E}$ for all $f\in \mathcal G(M)$.
\end{proposition}
\begin{proof} $f^0\circ S_{\tilde\E}\circ f^{-1}=f^0\circ\big( g^{(0)}\circ S_{\E}\circ g^{-1} \big)\circ f^{-1}=(f\circ g)^{(0)}\circ S_{\E}\circ (f\circ g)^{-1}$. 

Taking into account that $\tilde x$ and $\tilde y$ are differential invariants we get that $f\circ g= g$. This leads to the equality $f^0\circ S_{\tilde\E}\circ f^{-1}=S_{\tilde\E}$.
\end{proof}
\begin{corollary}
Differential equation \eqref{tildeE} is invariant w.r.t. the action of $\mathcal G(M)$ on differential equations \eqref{eq}.
\end{corollary}

%======================================================
\subsection{The algebra of all differential invariants} The {\it $n$-invariants principle}, see \cite{LchYum}, in our case of equations \eqref{eq} leads to the following result.
\begin{theorem}
Let an equation $\E$ of form \eqref{eq} be in general position in the domain $O'$ and $\tilde\E$ be its expression \eqref{tildeE} in coordinates $\tilde x, \tilde y$. Then the algebra of all differential invariants of  the equation $\E$ consist of all smooth  functions of differential invariants $\tilde x, \tilde y, \tilde J_3, \tilde J_2, \tilde J_1, \tilde J_0$, and their Tresse derivatives.
\end{theorem}

%=======================================================
\subsection{The equivalence problem} 
Let $\E_1$, $\E_2$ be equations of form \eqref{eq} in general position 
and $\tilde\E_1$, $\tilde\E_2$ be their forms in coordinates $(I_1\circ j_4S_{\E_1}$, $I_2\circ j_4S_{\E_1})$ and $(I_1\circ j_4S_{\E_2}$, $I_2\circ j_4S_{\E_2})$ respectively.  
 
\begin{theorem} Let $\E_1$ and $\E_2$ be equations of form \eqref{eq} in general positions in domains $O'_1\subset M$ and $O'_2\subset M$ respectively. Then there exists a diffeomorphism  $f\colon O'_1\rightarrow O'_2$ transforming $\E_1$ to $\E_2$ if and only if $\tilde \E_1=\tilde \E_2$.
\end{theorem}
\begin{proof} Suppose $\tilde \E_1=\tilde \E_2$, i.e. $S_{\tilde\E_2}= S_{\tilde\E_1}$. Then  $g_2^{(0)}\circ S_{\E_2}\circ g_2^{-1}=g_1^{(0)}\circ S_{\E_1}\circ g_1^{-1}$ or 
$S_{\E_2}= (g_2^{-1}\circ g_1)^{(0)}\circ S_{\E_1}\circ (g_2^{-1}\circ g_1)^{-1}$. Therefore,  the diffeomorphism $f$ transforming $\E_1$ to $\E_2$  is defined  by the formula $f=g_2^{-1}\circ g_1$.

Inversely, suppose that there is $f\in\mathcal{G}(M)$  transforming $\E_1$ to $\E_2$, i.e. $S_{\E_2}=f^{(0)}\circ S_{\E_1}\circ f^{-1}$. Then  $S_{\tilde\E_2}=g_2^{(0)}\circ S_{\E_2}\circ g_2^{-1}=g_2^{(0)}\circ (f^{(0)}\circ S_{\E_1}\circ f^{-1})\circ g_2^{-1}=
(g_2\circ f)^{(0)}\circ S_{\E_1}\circ (g_2\circ f)^{-1}
=(g_2\circ f)^{(0)}\circ \big((g_1^{-1})^{(0)}\circ  S_{\tilde\E_1}\circ g_1\big)\circ(g_2\circ f)^{-1}=(g_2\circ f\circ g_1^{-1})^{(0)}\circ S_{\tilde\E_1}\circ (g_2\circ f\circ g_1^{-1})^{-1}= S_{\tilde\E_1}$.
 The  last equality follows from Proposition \ref{InvrntSctn}. 
\end{proof}

%%%%%%%%%%%%%%%%%%%%%%%%%%%%%%%%%%%%%%%%%%

%
\end{document}